\documentclass[12pt, reqno]{amsart}
\usepackage{amsmath, amsthm, amscd, amsfonts, amssymb, mathrsfs, graphicx, color}
\usepackage[bookmarksnumbered, colorlinks, plainpages]{hyperref}

\newtheorem{thm}{Theorem}[section]
\newtheorem{cor}[thm]{Corollary}
\newtheorem{lem}[thm]{Lemma}

\theoremstyle{definition}
\newtheorem{defin}[thm]{Definition}
\newtheorem{rem}[thm]{Remark}

\numberwithin{equation}{section}

\newcommand{\la}{\left\langle}
\newcommand{\ra}{\right\rangle}
\newcommand{\nn}{\nonumber}

\begin{document}


\baselineskip=17pt


\title[Operator Entropy Inequalities]{Operator entropy inequalities}

\author[A. Morassaei, F. Mirzapour, M.S. Moslehian]{A. Morassaei$^1$, F. Mirzapour$^1$ and M. S. Moslehian$^2$}

\address{$^{1}$ Department of Mathematics, Faculty of Sciences, University of Zanjan, P. O. Box 45195-313, Zanjan, Iran.}
\email{morassaei@znu.ac.ir\\ f.mirza@znu.ac.ir}

\address{$^{2}$ Department of Pure Mathematics, Center of Excellence in Analysis on Algebraic Structures (CEAAS), Ferdowsi University of Mashhad, P. O. Box 1159, Mashhad 91775, Iran.}
\email{moslehian@um.ac.ir\\moslehian@member.ams.org}
\urladdr{\url{http://profsite.um.ac.ir/~moslehian/}}




\begin{abstract}
In this paper we investigate a notion of relative operator entropy,
which develops the theory started by J.I. Fujii and E. Kamei [Math.
Japonica 34 (1989), 341--348]. For two finite sequences
$\mathbf{A}=(A_1,\cdots,A_n)$ and $\mathbf{B}=(B_1,\cdots,B_n)$  of
positive operators acting on a Hilbert space, a real number $q$ and
an operator monotone function $f$ we extend the concept of entropy
by
$$
S_q^f(\mathbf{A}|\mathbf{B}):=\sum_{j=1}^nA_j^{\frac{1}{2}}\left(A_j^{-\frac{1}{2}}B_jA_j^{-\frac{1}{2}}\right)^qf\left(A_j^{-\frac{1}{2}}B_jA_j^{-\frac{1}{2}}\right)A_j^{\frac{1}{2}}\,,
$$
and then give upper and lower bounds for
$S_q^f(\mathbf{A}|\mathbf{B})$ as an extension of an inequality due
to T. Furuta [Linear Algebra Appl. 381 (2004), 219--235] under
certain conditions. Afterwards, some inequalities concerning the
classical Shannon entropy are drawn from it.
\end{abstract}

\maketitle

\section{Introduction and preliminaries}
Throughout the paper, let $\mathbb{B}(\mathscr{H})$ denote the
algebra of all bounded linear operators acting on a complex Hilbert
space $(\mathscr{H},\la \cdot,\cdot\ra)$ and $I$ is the identity
operator. In the case when $\dim \mathscr{H} = n$, we identify
$\mathbb{B}(\mathscr{H})$ with the full matrix algebra $\mathcal{M}_n(\mathbb{C})$
 of all $n \times n$ matrices with entries in the complex field $\mathbb{C}$ and denote its identity
 by $I_n$. A self-adjoint operator $A\in\mathbb{B}(\mathscr{H})$ is called positive if $\la Ax,x\ra\ge 0$ for all $x\in \mathscr{H}$ and then we write $A\ge0$. An operator $A$ is said to be strictly
positive (denoted by $A>0$) if it is positive and invertible. For
self-adjoint operators $A,B\in\mathbb{B}(\mathscr{H})$, we say $A\le
B$ if $B-A \ge 0$. Let $f$ be a continuous real valued function
defined on an interval $J$. The function $f$ is called operator
decreasing if $B\le A$ implies $f(A)\le f(B)$ for all
$A,B\in\mathbb{B}(\mathscr{H})$ with spectra in $J$. The function
$f$ is said to be operator concave on $J$ if
$$
\lambda f(A) + (1 - \lambda)f(B)\le f(\lambda A + (1 - \lambda)B)
$$
for all self-adjoint operators $A,B\in\mathbb{B}(\mathscr{H})$ with
spectra in $J$ and all $\lambda\in[0,1]$.

In 1850 Clausius [Ann. Physik (2) 79 (1850), 368--397, 500--524]
introduced the notion of entropy in the thermodynamics. Since then
several extensions and reformulations have been developed in various
disciplines; cf. \cite{ME, LR, L, NU}. There have been investigated
the so-called entropy inequalities by some mathematicians, see
\cite{BLP, BS, FUR2} and references therein. A relative operator
entropy of strictly positive operators $A,B$ was introduced in the
noncommutative information theory by Fujii and Kamei \cite{FK} by
$$
S(A|B) = A^{\frac{1}{2}} \log(A^{-\frac{1}{2}}BA^{-\frac{1}{2}})A^{\frac{1}{2}}.
$$
When $A$ is positive, one may set
$S(A|B):=\lim_{\epsilon\to+0}S(A+\epsilon I|B)$ if the limit which
is taken in the strong operator topology exists. In the same paper,
it is shown that $S(A|B) \le 0$ if $A \ge B$. There is an analogous
notion called the perspective function in the literature, see
\cite{E, CK}: If $f: [0,\infty) \to \mathbb{R}$ is an operator
convex function, then the perspective function $g$ associated to $f$
is defined by
\begin{eqnarray*}
  g(B,A)=A^{\frac{1}{2}}f(A^{-\frac{1}{2}}BA^{-\frac{1}{2}})A^{\frac{1}{2}}
\end{eqnarray*}
for any self-adjoint operator $B$ and any strictly positive operator
$A$. One may consider a more general case. Let
$\widetilde{B}=(B_1,\cdots,B_n)$ and
$\widetilde{A}=(A_1,\cdots,A_n)$ be $n$-tuples of self-adjoint and
strictly positive operators, respectively. Then the non-commutative
$f$-divergence functional $\Theta$ is defined by
\begin{eqnarray*}
  \Theta(\widetilde{B},\widetilde{A})=\sum_{i=1}^{n}A_i^{\frac{1}{2}}f(A_i^{-\frac{1}{2}}B_iA_i^{-\frac{1}{2}})A_i^{\frac{1}{2}}.
\end{eqnarray*}

Next, recall that $X\natural_q Y$ is defined by
$X^{\frac{1}{2}}\left(X^{-\frac{1}{2}}YX^{-\frac{1}{2}}\right)^qX^{\frac{1}{2}}$
for any real number $q$ and any strictly positive operators $X$ and
$Y$. For $p\in[0,1]$, the operator $X \natural_p Y$ coincides with
the well-known $p$-power mean of $X, Y$.

Furuta \cite{F} defined a parametric extension of the operator
entropy by
$$
S_p(A|B)=A^{\frac{1}{2}}\left(A^{-\frac{1}{2}}BA^{-\frac{1}{2}}\right)^p\log\left(A^{-\frac{1}{2}}BA^{-\frac{1}{2}}\right)A^{\frac{1}{2}}\,,
$$
where $p\in[0, 1]$ and $A, B$ are strictly positive operators on a
Hilbert space $\mathscr H$ and proved some operator entropy
inequalities as follows: if $\{A_1,\cdots,A_n\}$ and
$\{B_1,\cdots,B_n\}$ are two sequences of strictly positive
operators on a Hilbert space $\mathscr H$ such that
$\sum_{j=1}^nA_j\natural_p B_j\le I$., then
\begin{align}\label{e8}
& \log\left[\sum_{j=1}^n\left(A_j \natural_{p+1} B_j\right)+t_0\left(I-\sum_{j=1}^nA_j \natural_p B_j\right)\right]-(\log t_0)\left(I-\sum_{j=1}^nA_j \natural_p B_j\right)\nonumber\\
& \ge \sum_{j=1}^nS_p(A_j|B_j)\\
& \ge -\log\left[\sum_{j=1}^n\left(A_j \natural_{p-1}
B_j\right)+t_0\left(I-\sum_{j=1}^nA_j \natural_p
B_j\right)\right]\\
&\quad+(\log t_0)\left(I-\sum_{j=1}^nA_j \natural_p
B_j\right)\nonumber
\end{align}
for a fixed real number $t_0>0$.\\

The object of this paper is to state an operator entropy inequality
parallel to the main result of \cite{F} and refine some known
operator entropy inequalities.

\section{Operator entropy inequality}

The following notion is basic in our work.

\begin{defin}\label{d1}
Assume that $\mathbf{A}=(A_1,\cdots,A_n)$ and
$\mathbf{B}=(B_1,\cdots,B_n)$ are finite sequences of strictly
positive operators on a Hilbert space $\mathscr{H}$. For
$q\in\mathbb R$ and an operator monotone function $f: (0,\infty) \to
[0,\infty)$ the \emph{generalized operator Shannon entropy} is
defined by
\begin{equation}\label{ShEn}
S_q^f(\mathbf{A}|\mathbf{B}):=\sum_{j=1}^nS_q^f(A_j|B_j)\,,
\end{equation}
where $S^f_q(A_j|B_j)=A_j^{1/2}\left(A_j^{-1/2}B_jA_j^{-1/2}\right)^qf\left(A_j^{-1/2}B_jA_j^{-1/2}\right)A_j^{1/2}$.
\end{defin}
We recall that for $q=0$, $f(t)=\log t$ and $A, B>0$, we get the
relative operator entropy $S_0^f(A|B)=A^{\frac{1}{2}}\log
\left(A^{-\frac{1}{2}}BA^{-\frac{1}{2}}\right)A^{\frac{1}{2}}=S(A|B)$.
It is interesting to point out that $S_q(A|B)=-S_{1-q}(B|A)$ for any
real number $q$, in particular, $S_1(A|B) = - S(B|A)$. In fact,
since $Xf(X^*X)=f(XX^*)X$ holds for every $X\in\mathbb{B}(\mathscr
H)$ and every continuous function $f$ on the interval $[0,\|X\|^2]$,
considering $X=B^{1/2}A^{-1/2}$ and $f(t)=\log t$ we get
\begin{align*}
S_q(A|B) &=A^{\frac{1}{2}}\left(A^{-\frac{1}{2}}BA^{-\frac{1}{2}}\right)^q\log\left(A^{-\frac{1}{2}}BA^{-\frac{1}{2}}\right)A^{\frac{1}{2}}\\
&=B^{\frac{1}{2}}B^{-\frac{1}{2}}A^{\frac{1}{2}}\left(A^{-\frac{1}{2}}BA^{-\frac{1}{2}}\right)^q\log\left(A^{-\frac{1}{2}}BA^{-\frac{1}{2}}\right)
A^{\frac{1}{2}}B^{-\frac{1}{2}}B^{\frac{1}{2}}\\
&=B^{\frac{1}{2}}{X^*}^{-1}\left(X^*X\right)^q\log\left(X^*X\right)X^{-1}B^{\frac{1}{2}}\\
&=B^{\frac{1}{2}}{X^{-1}}^*\left(X^{-1}{X^{-1}}^*\right)^{-q}\log\left(X^*X\right)X^{-1}B^{\frac{1}{2}}\\
&=B^{\frac{1}{2}}\left({X^{-1}}^*X^{-1}\right)^{1-q}\left({X^{-1}}^*X^{-1}\right)^{-1}{X^{-1}}^*\log\left(X^*X\right)X^{-1}B^{\frac{1}{2}}\\
&=B^{\frac{1}{2}}\left({X^{-1}}^*X^{-1}\right)^{1-q}X\log\left(X^*X\right)X^{-1}B^{\frac{1}{2}}\\
&=B^{\frac{1}{2}}\left({X^{-1}}^*X^{-1}\right)^{1-q}\log\left(XX^*\right)XX^{-1}B^{\frac{1}{2}}\\
&=-B^{\frac{1}{2}}\left({X^{-1}}^*X^{-1}\right)^{1-q}\log\left({X^{-1}}^*X^{-1}\right)B^{\frac{1}{2}}\\
&=-B^{\frac{1}{2}}\left({X^*}^{-1}X^{-1}\right)^{1-q}\log\left({X^*}^{-1}X^{-1}\right)B^{\frac{1}{2}}\\
&=-S_{1-q}(B|A)\,.
\end{align*}

We need the following useful lemma.

\begin{lem}\cite[Proposition 3.1]{F}\label{HANP}
If $f$ is a continuous real function on an interval $J$, then the following conditions are equivalent:
\begin{enumerate}
\item[(i)] $f$ is operator concave.

\item[(ii)] $f(C^*XC+t_0(I-C^*C))\ge C^*f(X)C+f(t_0)(I-C^*C))$ for any operator $C$ with $\|C\|\le1$ and any self-adjoint operator $X$ with $sp(X)\subseteq J$ and for a fixed real number $t_0\in J$.

\item[(iii)] $f(\sum_{j=1}^nC_j^*X_jC_j+t_0(I-\sum_{j=1}^nC_j^*C_j))\ge \sum_{j=1}^nC_j^*f(X_j)C_j+f(t_0)(I-\sum_{j=1}^nC_j^*C_j))$ for operators $C_j$ with $\sum_{j=1}^nC_j^*C_j\le I$ and self-adjoint operators $X_j$ with $sp(X_j)\subseteq J$ for $j=1,\cdots,n$ and for a fixed real number $t_0\in J$.
\end{enumerate}
\end{lem}
For other equivalence conditions the reader may consult \cite{FMPS,
MOS} and references therein. Using an idea of \cite{F} we prove the
next result.

The following result gives lower and upper bounds for $S_q^f(\mathbf
A|\mathbf B)$.
\begin{thm}\label{tShEnIn}
Assume that $f$, $\mathbf A$ and $\mathbf B$ are as in Definition
\ref{d1}. Let $\sum_{j=1}^nA_j=\sum_{j=1}^nB_j=I$ and $f$ be
operator concave. Then
\begin{align}\label{ShEnIn}
f&\left[\sum_{j=1}^n(A_j\natural_{p+1}B_j)+t_0\left(I-\sum_{j=1}^nA_j\natural_pB_j\right)\right]-f(t_0)\left(I-\sum_{j=1}^nA_j\natural_pB_j\right)\nn\\
&\ge S_p^f(\mathbf{A}|\mathbf{B})
\end{align}
for all $p\in[0,1]$ and for any fixed real number $t_0>0$, and
\begin{align}\label{ShEnIn2}
-f&\left[\sum_{j=1}^n(A_j\natural_{p-1}B_j)+t_0\left(I-\sum_{j=1}^nA_j\natural_pB_j\right)\right]+f(t_0)\left(I-\sum_{j=1}^nA_j\natural_pB_j\right)\nn\\
&\le S_p^f(\mathbf{A}|\mathbf{B})
\end{align}
for all $p\in[2,3]$ and for any fixed real number $t_0>0$.
\end{thm}

\begin{proof}
Since $\sum_{j=1}^nA_j \natural_q B_j \le
\left(\sum_{j=1}^nA_j\right) \natural_q
\left(\sum_{j=1}^nB_j\right)$ (see \cite[Theorem 5.7]{FMPS}, for
every $q\in[0,1]$, and $\sum_{j=1}^nA_j=\sum_{j=1}^nB_j=I$, we have
$\sum_{j=1}^nA_j \natural_{p} B_j \le I$.

Let us fix a positive real number $t_0$. Since $f$ is operator concave, we get
\begin{align*}
f\Bigg[&\sum_{j=1}^n(A_j\natural_{p+1}B_j)+t_0\left(I-\sum_{j=1}^nA_j\natural_pB_j\right)\Bigg]\\
=&f\Bigg[\sum_{j=1}^n\left(\Big(A_j^{-\frac{1}{2}}B_jA_j^{-\frac{1}{2}}\Big)^{\frac{p}{2}}A_j^{\frac{1}{2}}\right)^*
\Big(A_j^{-\frac{1}{2}}B_jA_j^{-\frac{1}{2}}\Big)\left(\Big(A_j^{-\frac{1}{2}}B_jA_j^{-\frac{1}{2}}\Big)^{\frac{p}{2}}A_j^{\frac{1}{2}}\right)\\
&+t_0\left(I-\sum_{j=1}^nA_j\natural_p B_j\right)\Bigg]\\
\ge &\sum_{j=1}^nA_j^{\frac{1}{2}}\left(A_j^{-\frac{1}{2}}B_jA_j^{-\frac{1}{2}}\right)^{\frac{p}{2}}
f\left(A_j^{-\frac{1}{2}}B_jA_j^{-\frac{1}{2}}\right)\left(A_j^{-\frac{1}{2}}B_jA_j^{-\frac{1}{2}}\right)^{\frac{p}{2}}
A_j^{\frac{1}{2}}\\
&+f(t_0)\left(I-\sum_{j=1}^nA_j\natural_pB_j\right)\quad\quad\quad\quad\mbox{(by~the ~Lemma \ref{HANP} (iii))}\\
=&\sum_{j=1}^nA_j^{\frac{1}{2}}\left(A_j^{-\frac{1}{2}}B_jA_j^{-\frac{1}{2}}\right)^p
f\left(A_j^{-\frac{1}{2}}B_jA_j^{-\frac{1}{2}}\right)A_j^{\frac{1}{2}}+f(t_0)\left(I-\sum_{j=1}^nA_j\natural_pB_j\right)\\
=&\sum_{j=1}^nS_p^f(A_j|B_j)+f(t_0)\left(I-\sum_{j=1}^nA_j\natural_pB_j\right)\,,
\end{align*}
whence
\begin{align*}
f\Bigg[\sum_{j=1}^n(A_j\natural_{p+1}B_j)&+t_0\left(I-\sum_{j=1}^nA_j\natural_pB_j\right)\Bigg]\nn\\
\ge&\sum_{j=1}^nS_p^f(A_j|B_j)+f(t_0)\left(I-\sum_{j=1}^nA_j\natural_pB_j\right)\,,
\end{align*}
Following a similar argument, we obtain
\begin{align*}
f\Bigg[\sum_{j=1}^n(A_j\natural_{p-1}B_j)&+t_0\left(I-\sum_{j=1}^nA_j\natural_pB_j\right)\Bigg]\nn\\
\ge&\sum_{j=1}^nS_{p-2}^f(A_j|B_j)+f(t_0)\left(I-\sum_{j=1}^nA_j\natural_pB_j\right)\,.
\end{align*}
Thus
\begin{align*}
-f\Bigg[\sum_{j=1}^n(A_j\natural_{p-1}B_j)&+t_0\left(I-\sum_{j=1}^nA_j\natural_pB_j\right)\Bigg]+f(t_0)\left(I-\sum_{j=1}^nA_j\natural_pB_j\right)\nn\\
\le &-S_{p-2}^f(\mathbf{A}|\mathbf{B})\,.
\end{align*}
Since $f$ is a continuous nonnegative function, $X^qf(X) \ge 0$ for
every $X \ge 0$ and $q\in\mathbb R$. Hence
$$
\left(A_j^{-\frac{1}{2}}B_jA_j^{-\frac{1}{2}}\right)^qf\left(A_j^{-\frac{1}{2}}B_jA_j^{-\frac{1}{2}}\right) \ge 0\,.
$$
Consequently, $S_q^f(A_j|B_j) \ge 0$. Thus
\begin{align*}
S_p^f(A_j|B_j)+S_{p-2}^f(A_j|B_j) \ge 0\quad(j=1,\cdots,n)\,,
\end{align*}
whence $-S_{p-2}^f(\mathbf{A}|\mathbf{B}) \le
S_p^f(\mathbf{A}|\mathbf{B})$, which yields the required result.
\end{proof}
\begin{rem}
By taking $f(t)=\log t$ in Theorem \ref{tShEnIn}, we get \eqref{e8}.
\end{rem}
\begin{cor}\label{coei1}
Let $\mathbf{A}=(A_1,\cdots,A_n)$ and $\mathbf{B}=(B_1,\cdots,B_n)$
be two sequences of strictly positive operators on a Hilbert space
$\mathscr{H}$ such that $\sum_{j=1}^nA_j=\sum_{j=1}^nB_j=I$. If $f: (0,\infty) \to [0,\infty)$ is a function which
is both operator monotone and operator concave, then
\begin{enumerate}
\item[(i)] $f\Big(\sum_{j=1}^nB_jA_j^{-1}B_j\Big)\ge S_1^f(\mathbf{A}|\mathbf{B})$,
\item[(ii)] $f(I)\ge S_0^f(\mathbf{A}|\mathbf{B})$.
\end{enumerate}
\end{cor}
\begin{proof}
(i) Setting $p=1$ in Theorem \ref{tShEnIn} and applying
$\sum_{j=1}^n A_j \natural_1 B_j=\sum_{j=1}^n B_j=I$, we obtain
\begin{align*}
f\left(\sum_{j=1}^nB_jA_j^{-1}B_j\right)=f\left(\sum_{j=1}^n A_j
\natural_2 B_j\right) \ge S_1^f(\mathbf{A}|\mathbf{B})\,.
\end{align*}
(ii) Putting $p=0$ in Theorem \ref{tShEnIn} and using $\sum_{j=1}^n
A_j \natural_0 B_j=\sum_{j=1}^n A_j=I$, we get
\begin{align*}
f(I)=f\left(\sum_{j=1}^nB_j\right)=f\left(\sum_{j=1}^n A_j
\natural_1B_j\right)\ge S_0^f(\mathbf{A}|\mathbf{B})\,.
\end{align*}
\end{proof}
Next we extend the operator entropy for $n$ strictly positive
operators $A_1,\cdots,A_n\in\mathbb{B}(\mathscr H)$ and refine the
operator entropy inequality.
\begin{cor}\label{coei2}
Let $A_1,\cdots,A_n\in\mathbb{B}(\mathscr H)$ be a sequence of
strictly positive operators on a Hilbert space $\mathscr{H}$ such
that $\sum_{j=1}^nA_j=I$. Then
\begin{eqnarray}\label{mos1}
\log\left(\sum_{j=1}^nA_j^{-1}\right)\ge (\log n)
I-\frac{1}{n}\sum_{j=1}^n\log A_j\,.
\end{eqnarray}
\end{cor}
\begin{proof}
Taking $\mathbf{A}=(A_1,\cdots,A_n)$ and
$\mathbf{B}=(\frac{1}{n}I,\cdots, \frac{1}{n}I)$ and $f(t)=\log t$
in Corollary \ref{coei1} (i), we get
\begin{align*}
-2(\log n)I+\log \left(\sum_{j=1}^nA_j^{-1}\right)&=\log\left(\frac{1}{n^2}\sum_{j=1}^nA_j^{-1}\right)\\
& \ge S_1^{\log}(\mathbf A | \mathbf B)\\
& =\sum_{j=1}^n\frac{1}{n}A_j^{-\frac{1}{2}}\log\left(\frac{1}{n}A_j^{-1}\right)A_j^{\frac{1}{2}}\\
& =\sum_{j=1}^n\frac{1}{n}\log\left(\frac{1}{n}A_j^{-1}\right)\\
& =-\sum_{j=1}^n\frac{1}{n}\left((\log n)I+\log A_j\right)\\
& =-(\log n)I-\frac{1}{n}\sum_{j=1}^n\log A_j\,,
\end{align*}
which yields \eqref{mos1}.
\end{proof}
\begin{cor}[Operator Entropy Inequality]\label{toei}
Assume that $A_1,\cdots,A_n\in\mathbb{B}(\mathscr H)$ are positive invertible
operators satisfying  $\sum_{j=1}^nA_j=I$. Then
\begin{eqnarray*}
-\sum_{j=1}^nA_j\log A_j\le(\log n)I\,.
\end{eqnarray*}
\end{cor}
\begin{proof}
Letting $\mathbf{A}=(A_1,\cdots,A_n)$, $\mathbf{B}=\left(\frac{1}{n}I,\cdots,\frac{1}{n}I\right)$
and $f(t)=\log t$ in Corollary \ref{coei1} (ii), we get
\begin{align*}
0&=\log I\\
& \ge S_0^{\log}(\mathbf{A}|\mathbf{B})\\
&=\sum_{j=1}^nA_j^{\frac{1}{2}}\log\left(\frac{1}{n}A_j^{-1}\right)A_j^{\frac{1}{2}}\\
& = \sum_{j=1}^nA_j^{\frac{1}{2}}\left(-(\log n)I-\log A_j\right)A_j^{\frac{1}{2}}\\
& = -(\log n)\sum_{j=1}^nA_j-\sum_{j=1}^nA_j^{\frac{1}{2}}\left(\log
A_j\right)A_j^{\frac{1}{2}}\,.
\end{align*}
\end{proof}
\begin{rem}
Let $\mathbf{a}=(a_1,\cdots$ $,a_n)$ and
$\mathbf{b}=(b_1,\cdots,b_n)$ be $n$-tuples of positive real numbers
such that $\sum_{j=1}^na_j=\sum_{j=1}^nb_j=1$. Put
$A_i=[a_i]_{1\times 1}\in{\mathcal M}_1(\mathbb{C})$ and
$B_i=[b_i]_{1\times 1}\in{\mathcal M}_1(\mathbb{C})$. It follows
from  Corollary \ref{coei1} (ii) that $0
\ge\sum_{j=1}^na_j\log\frac{b_j}{a_j}\,,$ which is an entropy
inequality related to the Kullback–-Leibler relative entropy or
information divergence $ S(p,q)=\sum_{j=1}^np_j\log\frac{p_j}{q_j} $
with the convention $x \log x = 0$ if $x = 0$, and $x \log y =
+\infty$ if $ y = 0$ and $x\neq 0$; cf. \cite{KL}.
\end{rem}

\begin{thm}
Let $p\in[0, 1]$ and let $A, B$ be two strictly positive operator on
a Hilbert space $\mathscr H$ such that $A \natural_{p-2} B\le I$ and
$B^2\le A^2$. If $f: (0,\infty) \to [0,\infty)$ is a function which
is both operator monotone and operator concave, then
\begin{align}\label{e111}
f&\big(A \natural_{p+1} B + t_0\left(I-A \natural_p B\right)\big)-f(t_0)\left(I-A \natural_p B\right)\nn\\
&\ge S_p^f(A|B)\\
&\ge -f\big(A \natural_{p-1} B+t_0\left(I-A \natural_p
B\right)\big)+f(t_0)\left(I-A \natural_p B\right)\nn\,,
\end{align}
for a fixed real number $t_0>0$.
\end{thm}
\begin{proof}
It follows from $A \natural_{p-2} B\le I$ that
\begin{align*}
A^{\frac{1}{2}}\left(A^{-\frac{1}{2}}BA^{-\frac{1}{2}}\right)^{p-2}A^{\frac{1}{2}} &\le I\\
 \left(A^{-\frac{1}{2}}BA^{-\frac{1}{2}}\right)^{p-2} &\le A^{-1}\\
\left(A^{-\frac{1}{2}}BA^{-\frac{1}{2}}\right)^p &\le \left(A^{-\frac{1}{2}}BA^{-\frac{1}{2}}\right)A^{-1}\left(A^{-\frac{1}{2}}BA^{-\frac{1}{2}}\right)\\
 A^{\frac{1}{2}}\left(A^{-\frac{1}{2}}BA^{-\frac{1}{2}}\right)^pA^{\frac{1}{2}} &\le
 BA^{-2}B\,.
\end{align*}
Since $B^2 \le A^2$ and the map $t\mapsto-\frac{1}{t}$ is operator
monotone, we have
$$A^{\frac{1}{2}}\left(A^{-\frac{1}{2}}BA^{-\frac{1}{2}}\right)^pA^{\frac{1}{2}}\le
I$$ so that $A \natural_p B\le I$. Now the same reasoning as in the
proof of Theorem \ref{tShEnIn} (with $n=1$ and by using Lemma
\ref{HANP} (ii)) yields the desired inequalities.
\end{proof}
Recall that a map
$\Phi:\mathbb{B}(\mathscr{H})\to\mathbb{B}(\mathscr{K})$, where
$\mathscr{H}$ and $\mathscr{K}$ are Hilbert spaces, is called
positive if $\Phi(A)\geq 0$ whenever $A\geq 0$ and is said to be
normalized if it preserves the identity. The paper \cite[Lemma
5.2]{MMM2} includes a refinement
of the Jensen inequality for Hilbert space operators as follows:\\
Let $\mu=(\mu_1,\cdots,\mu_m)$ and
$\lambda=(\lambda_1,\cdots,\lambda_n)$ be two probability vectors.
By a (discrete) weight function (with respect to $\mu$ and
$\lambda$) we mean a mapping $\omega:\{(i,j)~:~1\le i\le m,~1\le
j\le n\}\to [0,\infty)$ such that
$\sum_{i=1}^m\omega(i,j)\mu_i=1\,\,(j=1,\cdots,n)$ and
$\sum_{j=1}^n\omega(i,j)\lambda_j=1\,\,(i=1,\cdots,m)$. If $f$ is a
real-valued operator concave function on an interval $J$, $A_1, \cdots, A_n$ are self-adjoint operators with spectra
in $J$ and
$\Phi:\mathbb{B}(\mathscr{H})\to\mathbb{B}(\mathscr{K})$ is a
normalized positive map, then
\begin{equation}\label{j1}
f\left(\sum_{j=1}^n\lambda_j\Phi(A_j)\right)\ge \sum_{i=1}^m\mu_i
f\left(\sum_{j=1}^n\omega(i,j)\lambda_j\Phi(A_j)\right)
\ge\sum_{j=1}^n\lambda_j\Phi(f(A_j))\,.
\end{equation}

A matrix $A=[a_{ij}]\in {\mathcal M}_n(\mathbb{C}) $ is said to be a
\emph{doubly stochastic matrix} if $a_{ij} \ge
0\quad(i,j=1,\cdots,n)$ and
$\sum_{i=1}^na_{ij}=\sum_{j=1}^na_{ij}=1$. Now we introduce a
refinement of the operator Jensen inequality.

\begin{thm}\label{t1}
Suppose that $f$ is a real-valued operator concave function on an
interval $J$ and $A_1, \cdots, A_n$ are self-adjoint operators with
spectra in an interval $J$. Assume that $B=[b_{ij}]$ and
$C=[c_{ij}]$ are two $n\times n$ doubly stochastic matrices,
$\omega_1$ and $\omega_2$ are weight functions with respect to the
same probability vector and
$\Phi:\mathbb{B}(\mathscr{H})\to\mathbb{B}(\mathscr{K})$ is a
normalized positive map. If the operator-valued functions
$F_{\omega_1,\omega_2}$ and $F_{B,C}$ are defined by
\begin{align*}
F_{\omega_1,\omega_2}(t)&:=\sum_{i=1}^m\mu_if\left(\sum_{j=1}^n
[(1-t)\omega_1(i,j)+t\omega_2(i,j)]\lambda_j\Phi(A_j)\right)\quad(0\le
t\le1)
\end{align*}
and
\begin{align}
F_{B,C}(t)&:=\frac{1}{n}\sum_{i=1}^nf\left(\sum_{j=1}^n[(1-t)b_{ij}+tc_{ij}]\Phi(A_j)\right)\quad(0\le
t\le1)\,,\label{e7}
\end{align}
then
\begin{enumerate}
\item[(i)]
\begin{equation}\label{e5}
f\left(\sum_{j=1}^n\lambda_j\Phi(A_j)\right)\ge
F_{\omega_1,\omega_2}(t)
\ge\sum_{j=1}^n\lambda_j\Phi(f(A_j))\quad(0\le t\le 1)\,.
\end{equation}
In particular,
$$
f\left(\frac{1}{n}\sum_{j=1}^n\Phi(A_j)\right)\ge
F_{B,C}(t)\ge\frac{1}{n}\sum_{j=1}^n\Phi(f(A_j))\quad(0\le t\le1)\,.
$$
\item[(ii)] For any $i\quad(i=1,\cdots,n)$, the maps
$$
t\longmapsto
f\left(\sum_{j=1}^n[(1-t)\omega_1(i,j)+t\omega_2(i,j)]\lambda_j\Phi(A_j)\right)
\quad(0\le t\le 1),
$$
as well as the function $F_{\omega_1,\omega_2}$ are operator
concave.\\ In particular, $F_{B,C} ~\emph{is concave on}~ [0,1]$.
\end{enumerate}
\end{thm}
\begin{proof}
(i) Since for every $t$ in $[0,1]$, the map
$$
(i,j)\longmapsto(1-t)\omega_1(i,j)+t\omega_2(i,j)\hspace{1cm}(1\le
i\le m,~1\le j\le n)
$$
is a weight function, \eqref{e5} follows from \eqref{j1}. By taking $m=n, \lambda_j=\mu_i=\frac{1}{n}, \omega_1(i,j)=nb_{ij}, \omega_2(i,j)=nc_{ij}$ in $F_{\omega_1,\omega_2}(t)$, we obtain the second part.\\
(ii) Let $\eta_1,\eta_2 \ge 0$ with $\eta_1+\eta_2=1$ and let
$t_1,t_2\in [0,1]$. For every $i$ with $1\le i\le m$, we have
\begin{align*}
f\Big(&\sum_{j=1}^n[(1-\eta_1 t_1-\eta_2 t_2)\omega_1(i,j)+(\eta_1
t_1+\eta_2 t_2)\omega_2(i,j)]
\lambda_j \Phi(A_j)\Big)\\
=&f\Big(\eta_1\sum_{j=1}^n[(1-t_1)\omega_1(i,j)+t_1\omega_2(i,j)]\lambda_j \Phi(A_j)\\
&+\eta_2\sum_{j=1}^n[(1-t_2)\omega_1(i,j)+t_2\omega_2(i,j)]\lambda_j \Phi(A_j)\Big)\\
\ge& \eta_1 f\Big(\sum_{j=1}^n[(1-t_1)\omega_1(i,j)+t_1\omega_2(i,j)]\lambda_j \Phi(A_j)\Big)\quad\mbox{(by the concavity of $f$)}\\
&+\eta_2
f\Big(\sum_{j=1}^n[(1-t_2)\omega_1(i,j)+t_2\omega_2(i,j)]\lambda_j
\Phi(A_j)\Big)\,,
\end{align*}
which implies (ii). The concavity of $F_{B,C}$ over $[0,1]$ is
clear.
\end{proof}
By taking $f(t)=-t\log t$ and $\Phi(t)=t$ in \eqref{e7} and by using
Theorem \ref{t1}, we obtain the following result:
\begin{cor}[Refinement of operator entropy inequality]\label{cor1}
Assume that $A_1, \cdots, A_n$ are positive self-adjoint invertible
operators with spectra in an interval $J$ and
$\sum_{j=1}^nA_j=I$. If $B=[b_{ij}]$ and $C=[c_{ij}]$ are two
$n\times n$ doubly stochastic matrices, then
\begin{align*}
(\log n)I\ge& \sum_{i=1}^n\left[-\left(\sum_{j=1}^n[(1-t)b_{ij}+tc_{ij}]A_j\right)\log\left(\sum_{j=1}^n[(1-t)b_{ij}+tc_{ij}]A_j\right)\right]\\
\ge&-\sum_{j=1}^nA_j\log A_j\quad\quad\quad(0\le t\le1)\,,
\end{align*}
\end{cor}


\end{document}